\documentclass[a4paper, 12pt]{amsart}
\usepackage[utf8]{inputenc}
\usepackage{amsthm,amsmath}
\usepackage{amsopn,amssymb}
\usepackage[dvipsnames]{xcolor}
\usepackage{fullpage,amsfonts,mathtools,mathrsfs}
\usepackage{aliascnt,enumerate,comment,xcolor,graphicx}
\usepackage{enumitem}
\usepackage{lineno} 
\usepackage[colorlinks=true, 
pdfstartview=FitV, linkcolor=blue, 
citecolor=blue, urlcolor=blue, backref = page]{hyperref}

\usepackage[msc-links, alphabetic]{amsrefs}

\usepackage[capitalise]{cleveref}
\usepackage{cancel}
\allowdisplaybreaks



\newtheorem{theorem}{Theorem}[section]

\newaliascnt{lemma}{theorem}
\newtheorem{lemma}[lemma]{Lemma}
\aliascntresetthe{lemma}

\newaliascnt{condition}{theorem}

\aliascntresetthe{condition}

\newaliascnt{corollary}{theorem}
\newtheorem{corollary}[corollary]{Corollary}
\aliascntresetthe{corollary}

\newaliascnt{proposition}{theorem}

\aliascntresetthe{proposition}

\theoremstyle{definition}
\newaliascnt{definition}{theorem}

\aliascntresetthe{definition}

\theoremstyle{remark}

\newaliascnt{remark}{theorem}

\aliascntresetthe{remark}

\allowdisplaybreaks

\numberwithin{equation}{section}

\DeclareMathOperator*{\E}{\mathbb{E}}

\usepackage{thmtools}
\usepackage[framemethod=TikZ]{mdframed}

\title{Partition regular linear equations over Sidon sets}
\author{Dev Ranjan Pandey}
\address{Department of Mathematics, Indian Institute of Science 
Education and Research Bhopal, Bhopal Bypass Road, Bhauri, Bhopal 462066, Madhya Pradesh, India}
\email{dev23@iiserb.ac.in}


\begin{document}

\begin{abstract}
  In this article, 
  we show that the Sidon subsets of \([N]^d\) are Fourier uniform, 
  and we prove a dense model lemma for Sidon sets.
  Using these results and a higher-dimensional version of Rado's theorem,
  we prove that given 
  any partition regular linear equation in \(s \geq 5\) variables
  with nonzero coefficients, 
  and any finite partition of a sufficiently dense Sidon subset \(S\) of \([N]^d\),
  the number of monochromatic solutions to this equation is 
  \(\gg |S|^s N^{-d}\) for all large \(N\).
  As a corollary of the Fourier uniformity result, 
  we show that dense Sidon subsets of \([N]^d\)
  are equidistributed in certain arithmetic and Bohr structures.
  Our proofs are motivated by the arguments of Ortega and Prendiville.
\end{abstract}

\maketitle

\section{Introduction} 

\subsection{Motivation}
A subset \(S\) of a semigroup is said to be a Sidon set 
if \(s_1 + s_2 = s_3 + s_4\) implies 
\(\left\{ s_1, s_2 \right\} = \left\{ s_3, s_4 \right\}\)
for every \((s_1, s_2, s_3, s_4) \in S^4\).
We write \([N]^d\) for the set \([1,N]^d \cap \mathbb{Z}^d\).
Singer \cite{Singer_B2} proved that 
given any \(\varepsilon > 0\) and for all sufficiently large \(N\), 
there exists a Sidon subset of \([N]\) 
of size at least \((1 - \varepsilon)N^{1/2}\).
Erd\H{o}s and Tur\'an \cite{Erdos-Turan_B2} proved that
the size of any Sidon set in \([N]\) is at most 
\(N^{1/2} + O(N^{1/4})\).
The upper bound for the size of Sidon sets in \([N]\)
was improved further in subsequent years
by several authors.
Lindstr\"om \cite{Lindstrom_B2_sequences} proved that 
the size of any Sidon set in \([N]^d\) is at most 
\(N^{d/2} + O(N^{d^2/(2d+2)})\).
Cilleruelo \cite{Cilleruelo_Sidon_sets_in_N^d} extended the result of Lindstr\"om 
and proved that given any positive integers \(N_1, \dots, N_d\),
the size of any Sidon subset of \([N_1] \times \dots \times [N_d]\)
is at most \(N^{1/2} + O(N^{d/(2d+2)})\), 
where \(N = N_1 \cdots N_d\).
Also, he proved that there exists a Sidon subsets
of \([N]^d\) of size 
\((1 + o(1))N^{d/2}\).
A Sidon subset of \([N]^d\) whose size is close to the maximum possible size 
is called an \emph{extremal} Sidon set.
Erd\H{o}s and Freud \cite[Lemma 1]{Erdos-Freud_Sidon-sum}
proved that extremal Sidon subsets of \([N]\) 
are uniformly distributed in the interval \([1,N]\).
Given \(\delta > 0\), we say a Sidon subset \(S\) of \([N]^d\) 
to be \emph{\(\delta\)-dense} if
\(|S| \geq \delta N^{d/2}\).
Ortega and Prendiville \cite{OrtegaPrendiville_extremal-Sidon-sets}
proved that dense Sidon subsets of \([N]\)
are equidistributed in arithmetic progressions.
Moreover, they
proved that dense Sidon subsets of \([N]\) are equidistributed in 
smooth Bohr neighborhoods and regular Bohr sets.

A matrix \(A \in M_{t \times s}(\mathbb{Z})\) is said to satisfy 
the \emph{columns condition} if there exists a partition 
\(I_1, \dots, I_k\) of the column vectors of \(A\)
such that the following hold.
\begin{enumerate}
  \item \(\sum_{a \in I_1} a = 0\).
  \item \(\sum_{a \in I_j} a 
    \in \mathrm{span}_{\mathbb{Q}}(I_1 \cup \dots \cup I_{j-1})\)
    for every \(2 \leq j \leq k\).
\end{enumerate}
For a given matrix \(A\) with integer entries, 
the system of linear equations \(A \mathbf{x} = 0\) is said to be 
\emph{partition regular} if given any finite partition of the integers, 
the system \(A \mathbf{x} = 0\) admits a monochromatic solution.  
A system of linear equations \(A \mathbf{x} = 0\) is partition regular 
if and only if \(A\) satisfies the columns condition, see 
\cite[p.~250]{FGR_Roth}.
Given a matrix \(A \in M_{t \times s}(\mathbb{Z})\) of rank \(t\)
satisfying the columns condition,
and an \(r\)-coloring of the integers,
Frankl, Graham, and R\"odl \cite[Theorem~1]{FGR_Roth} proved
that there are at least \(c N^{s-t}\) many monochromatic solutions 
to the system \(A \mathbf{x} = 0\),
where \(c\) is a positive constant depending on \(A\) and \(r\).
Ortega and Prendiville \cite{OrtegaPrendiville_extremal-Sidon-sets} 
proved a Fourier uniformity property and a dense model lemma 
for Sidon subsets of \([N]\), and combining it with 
\cite[Theorem~1]{FGR_Roth} 
they obtained the following result.

\begin{theorem} [Ortega--Prendiville]
  Given positive integers \(r\) and \(s \geq 5\),
  any linear equation \(a_1 x_1 + \dots + a_s x_s = 0\)
  with nonzero coefficients and satisfying the columns condition
  admits \(\gg_{a_i, r} |S|^s/N\) many monochromatic solutions in any
  \(r\)-coloring of an extremal Sidon subset \(S\) of \([N]\).
\end{theorem}

Freschi, Hancock, and Treglown \cite{Freschi_ramsey-properties-primes}, 
and Robertson \cite{Robertson_multidimensional_Rado}
independently proved the following higher dimensional 
analog of \cite[Theorem~1]{FGR_Roth}.

\begin{theorem} \label{Robertson higher dimensional Rado}
  Let \(r, d\) be positive integers and 
  \(A \in M_{t \times s}(\mathbb{Z})\) be a matrix of rank \(t\) 
  satisfying the columns condition.
  Then there exist positive constants \(c, N_0\),
  both depending on \(r, d\) and \(A\), such that for every \(N \geq N_0\) 
  and every \(r\)-coloring of \([N]^d\)
  there are at least \(c N^{d(s-t)}\) many monochromatic solutions 
  to the system \(A \mathbf{x} = 0\).
\end{theorem}

\subsection{Results obtained}

In this article, given any positive integer \(d\),
we prove a Fourier uniformity result and
a dense model lemma for Sidon subsets of \([N]^d\).
Using these results we establish a partition regularity result for 
sufficiently dense Sidon subsets of \([N]^d\).
Also, we show that dense Sidon subsets of \([N]^d\) are 
equidistributed in certain arithmetic and Bohr structures.
The ideas of the proofs of these results are motivated by the arguments of
Ortega and Prendiville \cite{OrtegaPrendiville_extremal-Sidon-sets}.
We prove the following result, which establishes a
\(d\)-dimensional analog of \cite[Theorem~1.2]{OrtegaPrendiville_extremal-Sidon-sets}.
It shows that 
every Sidon subset of \([N]^d\) is Fourier uniform.

\begin{theorem}[Fourier uniformity] 
\label{Fourier uniformity}
  Every Sidon subset \(S\) of \([N]^d\) satisfies 
  \[
    \left\| 
      \widehat{\mathbf{1}}_S -
      \frac{|S|}{N^d} \widehat{\mathbf{1}}_{[N]^d}
    \right\|_\infty
    \ll_d
    N^{d/2}
    \left(  
      \left| \frac{|S|}{N^{d/2}} - 1 \right| + 
      {N^{-\frac{d}{2(1+2d)}}}
    \right)^{1/2}.
  \]
\end{theorem}

As a consequence of the above result,
we obtain 
in \cref{Equidistribution in intervals} that
dense Sidon subsets of \([N]^d\) are 
equidistributed in the set 
\(P_1 \times \cdots \times P_d\) where \(P_i\) is an arithmetic progression in
\(\mathbb{Z}\),
and 
in \cref{Equidistribution in regular Bohr sets} that
dense Sidon subsets of \([N]^d\) are 
equidistributed
in regular Bohr subsets of \([N]^d\).
Let \(\nu\) be a nonnegative function on \([N]^d\) and 
the Fourier uniform norm of \(\nu - \mathbf{1}_{[N]^d}\) is small.
The dense model lemma in \cref{Dense model lemma} states that  
for every nonnegative functions \(f_i\) with \(f_1 + \dots + f_r \leq \nu\)
there exist nonnegative functions \(g_1, \dots, g_r\) 
with \(g_1 + \dots + g_r \leq \mathbf{1}_{[N]^d}\)
such that the Fourier uniform norm of \(f_i - g_i\) is small 
for all \(1 \leq i \leq r\).

Using the above Fourier uniformity result, 
the dense model lemma,
and \cref{Robertson higher dimensional Rado}, we prove the following
\(d\)-dimensional analog of \cite[Theorem~1.15]{OrtegaPrendiville_extremal-Sidon-sets}.

\begin{theorem}[Partition regularity over Sidon sets] \label{Partition regularity}
  Let \(s \geq 5\) and \(d\) be positive integers and 
  \(a_1, \dots, a_s\) be nonzero integers with 
  \(\sum_{i \in I} a_i = 0\) for some nonempty subset \(I \subseteq [s]\).
  Then for every positive integer \(r\) there exist positive constants 
  \(c,N_0\) and \(\delta < 1\), all depending on 
  \(d, r, s, a_1, \dots, a_s\),
  such that for every \(N \geq N_0\) and every Sidon subset \(S\) of \([N]^d\),
  at least one of the following holds.
  \begin{enumerate}
    \item \(\displaystyle 
      |S| \leq \delta N^{d/2}.\)
    \item For any \(r\)-coloring of \(S\) there exists a color class \(C\) such that 
      \[
        \sum_{a_1 x_1 + \dots + a_s x_s = 0}
        \mathbf{1}_C(x_1) \cdots \mathbf{1}_C(x_s) 
        \geq 
        c |S|^s N^{-d}.
      \]
  \end{enumerate}
\end{theorem}

\subsection{Outline}

In \cref{Sec: Preliminaries}, we introduce required notations and state 
some preliminary results which are used in the subsequent sections.

In \cref{Sec: Fourier uniformity},
we state a higher dimensional version of
the van der Corput inequality in \cref{van der Corput differencing},
and provide a proof for completeness.
A lower bound for the sum of Fej\'er weights
over the nonzero elements of the difference set 
of a Sidon subset of \([N]^d\) is obtained in 
\cref{Lower bound for Fejer}.
Using these results, 
we establish in \cref{Fourier uniformity} that
Sidon subsets of \([N]^d\) are Fourier uniform.
In particular, given any Sidon subset \(S\) of \([N]^d\),
we obtain an upper bound for the Fourier uniform norm of the difference
between \(\mathbf{1}_S\) and a normalization of \(\mathbf{1}_{[N]^d}\)
by estimating a cross-correlation between the functions 
\(f_1\) and \(f_2\), for all 
\(f_1, f_2 \in \{ \mathbf{1}_S, \mathbf{1}_{[N]^d} \}\).

In \cref{Sec: Dense model lemma}, we prove a dense model lemma 
stated in \cref{Dense model lemma}.
It shows that the Fourier uniform norm of the difference
of \(\mathbf{1}_{[N]^d}\) and 
a nonnegative function supported on \([N]^d\) is large,
under certain condition on the Fourier uniform norm 
of the decomposition of these functions.
We define a norm \(\left\| \cdot \right\|^*\) for functions on \([N]^d\) and 
state some properties of this norm in \cref{properties of *norm}.
To prove \cref{Dense model lemma}, 
we construct a nonempty, closed and convex set
and apply a hyperplane separation theorem, stated in
\cref{Supporting hyperplane theorem}, to obtain certain functions
on \([N]^d\) with suitable norm properties.
Using the properties of \(*\)-norm 
and the Stone--Weierstrass theorem, 
we conclude our proof of the dense model lemma.

We prove \cref{Partition regularity}
in \cref{Sec: Partition regularity}.
We first prove \cref{Counting monochromatic solutions} 
using the pigeonhole principle and 
\cref{Robertson higher dimensional Rado}. 
It provides a lower bound for the weighted count of the number of
solutions to a given partition regular linear equation,
where the weight function is a certain component of a given 
decomposition of \(\mathbf{1}_{[N]^d}\) into nonnegative functions.
Finally, we prove \cref{Partition regularity} which shows that 
for all large \(N\), any finite coloring of a \(\delta\)-dense Sidon subset of \([N]^d\)
with sufficiently large \(\delta < 1\),
contains many monochromatic solutions to any given 
partition regular linear equation in at least five variables
and with nonzero coefficients.
To prove \cref{Partition regularity},
we assume that 
\(N\) is a sufficiently large integer, 
and 
for some \(r\)-coloring of a Sidon subset \(S\) of \([N]^d\), 
a given partition regular linear equation 
in at least five variables and with nonzero coefficients 
does not admit many monochromatic solutions in \(S\). 
Then using \cref{Counting monochromatic solutions,Fourier uniformity,Dense model lemma}
we show that there exists \(0 < \delta < 1\), 
independent of \(N\), such that
the Sidon set \(S\) is not \(\delta\)-dense.

In \cref{Sec: Equidistribution of dense Sidon sets},
we show some equidistribution properties of dense Sidon subsets of \([N]^d\).
Given a Sidon subset \(S\) of \([N]^d\) and a translation \(P\) of an axis-aligned rectangular 
sublattice of \(\mathbb{Z}^d\), we show in \cref{uniformity of Sidon in progression},
using \cref{Fourier uniformity}, that 
the difference of \(|S \cap P|\)
and \(|S| |P \cap [N]^d| / N^d\) is small.
Further, using \cref{Fourier uniformity} we prove that 
dense Sidon subsets of \([N]^d\) are equidistributed in regular Bohr sets.

\section{Notations and preliminaries} 
\label{Sec: Preliminaries}
Throughout the article, \(d\) and \(N\) denote positive integers.
The set 
\[
  \left\{  
    (n_1, \dots, n_d) \in \mathbb{Z}^d : 
    1 \leq n_i \leq N \text{ for all } 1 \leq i \leq d
  \right\}
\]
is denoted by \([N]^d\).
The additive group \(\mathbb{R}/\mathbb{Z}\) is denoted by \(\mathbb{T}\),
and we identify \(\mathbb{T}\) with the interval \([0,1)\)
equipped with addition modulo \(1\).
We consider \(\mathbb{T}\) to be equipped with the Haar probability measure.
For any elements \(n \coloneq (n_1, \dots, n_d) \in \mathbb{Z}^d\)
and \(\alpha \coloneq (\alpha_1, \dots, \alpha_d) \in \mathbb{T}^d\) 
we write 
\[
  \left\| \alpha \right\|
  \coloneq
  \max_{i} \min \left\{ \alpha_i, 1 - \alpha_i \right\}
\]
and
\[
  n \cdot \alpha 
  \coloneq
  n_1 \alpha_1 + \dots + n_d \alpha_d.
\]
For any \(x \in \mathbb{R}\) we denote \(e^{2 \pi i x}\) by \(e(x)\).

Let \(f,g\) be finitely supported complex-valued functions on \(\mathbb{Z}^d\).
The convolution of \(f\) and \(g\) is defined as 
\[
  f \ast g (x)
  \coloneq
  \sum_{n \in \mathbb{Z}^d}
  f(n) g(x - n),
\]
the Fourier transform \(\widehat{f} : \mathbb{T}^d \to \mathbb{C}\) 
of \(f\) is defined as 
\[
  \widehat{f}(\alpha)
  \coloneq
  \sum_{n \in \mathbb{Z}^d}
  f(n) e(n \cdot \alpha),
\]
and for any real number \(p \geq 1\), the \(p\)-norm of \(f\) is defined as 
\[
  \left\| f \right\|_p
  \coloneq
  \left(  
    \sum_{n \in \mathbb{Z}^d}
    \left| f(n) \right|^p
  \right)^{1/p}.
\]
The Fourier transform of an
integrable function \(h : \mathbb{T}^d \to \mathbb{C}\) is given by 
\[
  \widehat{h}(n) 
  \coloneq
  \int_{\mathbb{T}^d}
  h(\alpha) e(-n \cdot \alpha) 
  \mathrm{d} \alpha.
\]
We denote the supremum norm by \(\left\| \cdot \right\|_\infty\).

Combining the results of Cilleruelo 
\cite[Theorem 1, Theorem 5]{Cilleruelo_Sidon_sets_in_N^d}
implies the following result about Sidon sets.

\begin{lemma} \label{size of Sidon sets}
  Let \(F(N_1, \dots, N_d)\) denote the maximum possible size 
  of a Sidon subset of the rectangle \([N_1] \times \dots \times [N_d]\).
  Then we have 
  \[
    F(N_1, \dots, N_d) = (1 + o(1))(N_1 \cdots N_d)^{1/2}.
  \]  
\end{lemma}
 
\section{Fourier uniformity of Sidon sets} 
\label{Sec: Fourier uniformity}

For any positive integer \(H\), consider the Fej\'er weight function 
\(\mathcal{F}_H^d : \mathbb{Z}^d \to [0,1]\) defined as
\[
  \mathcal{F}_H^d(h) 
  \coloneq
  \frac{(\mathbf{1}_{[H]^d} \ast \mathbf{1}_{-[H]^d})(h)}{H^{2d}}.
\]
Observe that for any \(h \coloneq (h_1, \dots, h_d) \in \mathbb{Z}^d\) we have 
\[
  \mathcal{F}_H^d(h) 
  = 
  \prod_{i=1}^{d}
  \frac{(\mathbf{1}_{[H]} \ast \mathbf{1}_{-[H]})(h_i)}{H^{2}}
  = 
  \begin{cases} \displaystyle
    \frac{1}{H^d} 
    \prod_{i=1}^{d}
    \left(  
      1 - \frac{|h_i|}{H}
    \right) 
    &\text{if } h \in [-H,H]^d\\
    0 &\text{otherwise}.
  \end{cases}
\]

We state the following higher-dimensional version of the van~der~Corput inequality,
and provide a proof for the sake of completeness.

\begin{lemma}[van~der~Corput differencing] \label{van der Corput differencing}
  Let \(1 \leq H \leq N\) be integers and \(f : \mathbb{Z}^d \to \mathbb{C}\)
  be a function supported inside \([N]^d\). Then 
  \[
    \left| \sum_{n \in \mathbb{Z}^d} f(n) \right|^2 
    \leq
    \left( N + H \right)^d
    \sum_{h \in \mathbb{Z}^d} \mathcal{F}_H^d(h)
    \sum_{n \in \mathbb{Z}^d} 
    f(n) \overline{f(n + h)}.
  \]
\end{lemma}

\begin{proof}
  Using the Cauchy--Schwarz inequality and applying a change of variable
  observe that 
  \begin{align*}
    \left| \sum_{n \in \mathbb{Z}^d} f(n) \right|^2
    &=
    \frac{1}{H^{2d}}
    \left| \sum_{h \in [H]^d} \sum_{n \in [N + H]^d} f(n - h) \right|^2 \\
    &\leq
    \frac{(N + H)^d}{H^{2d}}
    \sum_{n \in [N + H]^d} \left| \sum_{h \in [H]^d} f(n - h) \right|^2 \\
    &=
    \frac{(N + H)^d}{H^{2d}}
    \sum_{n \in [N + H]^d} \sum_{h, h' \in [H]^d} f(n - h) \overline{f(n - h')} \\
    &=
    \frac{(N + H)^d}{H^{2d}}
    \sum_{h, h' \in [H]^d} \sum_{n \in [N]^d} f(n) \overline{f(n + h - h')} \\
    &=
    \frac{(N + H)^d}{H^{2d}}
    \sum_{h \in [-H,H]^d} 
    \left( \mathbf{1}_{[H]^d} \ast \mathbf{1}_{-[H]^d} \right)(h)
    \sum_{n \in [N]^d} f(n) \overline{f(n + h)} \\
    &=
    \left( N + H \right)^d
    \sum_{h \in \mathbb{Z}^d} \mathcal{F}_H^d(h)
    \sum_{n \in \mathbb{Z}^d} 
    f(n) \overline{f(n + h)}.
  \end{align*}
  Hence, \cref{van der Corput differencing} is proved.
\end{proof}
The following result yields a lower bound for \(\sum_{h} \mathcal{F}_H^d(h)\)
when the sum is taken over the difference set of a Sidon subset of \([N]^d\).

\begin{lemma} \label{Lower bound for Fejer}
  For any integers \(1 \leq H \leq N\), 
  every nonempty Sidon subset \(S\) of \([N]^d\) satisfies 
  \[
    \sum_{h \in (S-S) \setminus \left\{ 0 \right\}}
    \mathcal{F}_H^d(h)
    \geq 
    \frac{|S|^2}{(N + H)^d} - 
    \frac{|S|}{H^d}.
  \]
\end{lemma}

\begin{proof}
  Let \(1 \leq H \leq N\) be integers and 
  \(S \subseteq [N]^d\) be a nonempty Sidon set. 
  Note that 
  \[
    (\mathbf{1}_S \ast \mathbf{1}_{-S})(h)
    =
    \begin{cases}
      |S| &\text{if } h = 0 \\
      1 &\text{if } h \in (S-S) \setminus \left\{ 0 \right\} \\
      0 &\text{otherwise}.
    \end{cases} 
  \]
  It follows that 
  \begin{align*}
    \sum_{h \in (S-S) \setminus \left\{ 0 \right\}}
    \mathcal{F}_H^d(h)
    &=
    \sum_{h \in \mathbb{Z}^d}
    \mathcal{F}_H^d(h) (\mathbf{1}_S \ast \mathbf{1}_{-S})(h) - 
    \mathcal{F}_H^d(0) |S| \\
    &= 
    \frac{1}{H^{2d}} 
    \sum_{h \in \mathbb{Z}^d}
    (\mathbf{1}_{[H]^d} \ast \mathbf{1}_{-[H]^d})(h)
    (\mathbf{1}_S \ast \mathbf{1}_{-S})(h) - 
    \frac{|S|}{H^d} \\
    &= 
    \frac{1}{H^{2d}} 
    \sum_{h \in \mathbb{Z}^d}
    \left( (\mathbf{1}_S \ast \mathbf{1}_{[H]^d})(h) \right)^2 - 
    \frac{|S|}{H^d}.
  \end{align*}
  By the Cauchy--Schwarz inequality
  \begin{align*}
    \left( 
      \sum_{h \in \mathbb{Z}^d}
      (\mathbf{1}_S \ast \mathbf{1}_{[H]^d})(h) 
    \right)^2
    &\leq 
    \left( \sum_{h \in S + [H]^d} 1  \right)
    \left( 
      \sum_{h \in \mathbb{Z}^d}
      \left( (\mathbf{1}_S \ast \mathbf{1}_{[H]^d})(h) \right)^2  
    \right) \\
    &\leq 
    \left( N + H  \right)^d
    \sum_{h \in \mathbb{Z}^d}
    \left( (\mathbf{1}_S \ast \mathbf{1}_{[H]^d})(h) \right)^2.
  \end{align*}
  Observing that 
  \[
    \sum_{h \in \mathbb{Z}^d}
    (\mathbf{1}_S \ast \mathbf{1}_{[H]^d})(h) 
    =
    |S| H^d
  \]
  proves the result.
\end{proof}

Now we establish \cref{Fourier uniformity} using 
\cref{Lower bound for Fejer,van der Corput differencing}. 
It shows that the Fourier uniform norm of the difference of normalized
indicator function of \([N]^d\) and that of a Sidon subset of \([N]^d\) 
is bounded, in terms of \(N\) and the difference between 
the size of the considered Sidon set and \(N^{d/2}\).

\begin{proof} [Proof of \cref{Fourier uniformity}]
  Let \(S \subseteq [N]^d\) be a nonempty Sidon set. Take the functions 
  \begin{align*}
    f_1 \coloneq \mathbf{1}_S
    \text{\quad and \quad}
    f_2 \coloneq \frac{|S|}{N^d} \mathbf{1}_{[N]^d}.
  \end{align*}
  Let \(H \leq N\) be a positive integer.
  The triangle inequality and 
  \cref{van der Corput differencing} implies that 
  \begin{align*}
    \left\|  
      \widehat{f}_1 - \widehat{f}_2
    \right\|_\infty^2
    &=
    \sup_{\alpha \in \mathbb{T}^d}
    \left| 
      \sum_{n \in \mathbb{Z}^d} \left( f_1 - f_2 \right)(n) 
      e(n \cdot \alpha)
    \right|^2 \\
    &\leq
    \left( N + H \right)^d
    \sum_{h \in \mathbb{Z}^d} \mathcal{F}_H^d(h)
    \left|  
      \sum_{n \in \mathbb{Z}^d} 
      \left( f_1(n) - f_2(n) \right) \left( f_1(n + h) - f_2(n + h) \right)
    \right| \\
    &\leq
    \left( N + H \right)^d
    \sum_{h \in \mathbb{Z}^d} \mathcal{F}_H^d(h)
    \sum_{i,j \in \left\{ 1,2 \right\}}
    \left|  
      \frac{|S|^2}{N^d} - \sum_{n \in \mathbb{Z}^d} f_i(n) f_j(n + h)
    \right|.
  \end{align*}
  Let \(h \coloneq (h_1, \dots, h_d) \in \mathbb{Z}^d\) with \(|h_i| \leq H\)
  for all \(1 \leq i \leq d\). 
  Note that \([N]^d \setminus ([N]^d - h)\) can be written as the union of 
  \(d\) rectangles of size at most \(H N^{d-1}\).  
  Cilleruelo's upper bound for Sidon sets \cite[Theorem 1]{Cilleruelo_Sidon_sets_in_N^d} 
  implies that there exists a constant \(C' \geq 1\) depending on \(d\) 
  such that 
  \[
    |S| \leq C' N^{d/2}
    \quad \text{and} \quad
    \left| S \setminus ([N]^d - h) \right|
    \leq
    C' d (H N^{d-1})^{1/2}.
  \] 
  Consequently, we obtain that
  \begin{align} \label{f_1.f_2}
    \left| \frac{|S|^2}{N^d} - \sum_{n \in \mathbb{Z}^d} f_1(n) f_2(n + h) \right|
    =
    \frac{|S|}{N^{d}}
    \sum_{n \in \mathbb{Z}^d} 
    \left( \mathbf{1}_S - \mathbf{1}_{S \cap ([N]^d - h)} \right)(n)
    \leq 
    \frac{C' d H^{1/2} |S|}{N^{(d+1)/2}}
  \end{align}
  holds.
  Similarly, 
  \begin{align} \label{f_2.f_1}
    \left| \frac{|S|^2}{N^d} - \sum_{n \in \mathbb{Z}^d} f_2(n) f_1(n + h) \right|
    =
    \frac{|S|}{N^{d}}
    \sum_{n \in \mathbb{Z}^d} 
    \left( \mathbf{1}_S - \mathbf{1}_{S \cap ([N]^d + h)} \right)(n) 
    \leq 
    \frac{C' d H^{1/2} |S|}{N^{(d+1)/2}}
  \end{align}
  and
  \begin{align} \label{f_2.f_2}
    \left| \frac{|S|^2}{N^d} - \sum_{n \in \mathbb{Z}^d} f_2(n) f_2(n + h) \right|
    =
    \frac{|S|^2}{N^{2d}}
    \sum_{n \in \mathbb{Z}^d} 
    \left( \mathbf{1}_{[N]^d} - \mathbf{1}_{[N]^d \cap ([N]^d - h)} \right)(n)
    \leq 
    \frac{d H |S|^2}{N^{d+1}}
  \end{align}
  hold.
  Using \cref{Lower bound for Fejer} we deduce that 
  \begin{align} \label{f_1.f_1}
  \begin{aligned}
    \sum_{h \in \mathbb{Z}^d} \mathcal{F}_H^d(h)
    \left|  
      1 - \sum_{n \in \mathbb{Z}^d} f_1(n) f_1(n + h)
    \right|
    &=
    \sum_{h \in \mathbb{Z}^d} \mathcal{F}_H^d(h)
    \left| 1 - (\mathbf{1}_S \ast \mathbf{1}_{-S})(-h) \right|\\
    &=
    \frac{|S| - 1}{H^{d}} + 
    \sum_{h \in \mathbb{Z}^d \setminus \left\{ 0 \right\}} \mathcal{F}_H^d(h)
    \left( 1 - \mathbf{1}_{S-S}(h) \right) \\
    &\leq
    1 - \frac{|S|^2}{(N + H)^d} + 
    \frac{2|S|}{H^{d}} .
  \end{aligned}
  \end{align}
  Using \eqref{f_1.f_2}, \eqref{f_2.f_1}, \eqref{f_2.f_2} and \eqref{f_1.f_1}
  we get that
  \begin{multline*}
    \left\| 
      \widehat{\mathbf{1}}_S -
      \frac{|S|}{N^d} \widehat{\mathbf{1}}_{[N]^d}
    \right\|_\infty^2 \\
    \leq
    \left( N + H \right)^d
    \left(  
      1 - \frac{|S|^2}{(N + H)^d} +
      \frac{2|S|}{H^{d}} +
      \left| \frac{|S|^2}{N^d} - 1 \right| + 
      \frac{2C' d H^{1/2} |S|}{N^{(d+1)/2}} +
      \frac{d H |S|^2}{N^{d+1}} 
    \right).
  \end{multline*}
  Thus, taking \(H = \lfloor N^{\frac{1+d}{1+2d}} \rfloor\) 
  and using \(|S| \leq C' N^{d/2}\) yields that
  \begin{align*}
    &\left\| 
      \widehat{\mathbf{1}}_S -
      \frac{|S|}{N^d} \widehat{\mathbf{1}}_{[N]^d}
    \right\|_\infty^2 \\
    &\leq
    (2N)^d
    \left(  
      1 - \frac{|S|^2}{N^d + 2^d N^{d-1} H} +
      \left| \frac{|S|^2}{N^d} - 1 \right| +
      \frac{2C'^2d}{N^{\frac{d}{2(1+2d)}}} +
      \frac{C'^2d}{N^{\frac{d}{1+2d}}} + 
      \frac{2^{d+1} C'}{N^{\frac{d}{2(1+2d)}}} 
    \right)\\
    &\leq
    (2N)^d
    \left(  
      \frac{2^d N^{d-1} H}{N^d + 2^d N^{d-1} H} +
      \frac{N^d - |S|^2}{N^d + 2^d N^{d-1} H} +
      \left| \frac{|S|^2 - N^d}{N^d} \right| +
      \frac{2^{d+3} C'^2d}{N^{\frac{d}{2(1+2d)}}} 
    \right)\\
    &\leq
    (2N)^d
    \left(  
      \frac{4C' \left| |S| - N^{d/2} \right|}{N^{d/2}} +
      \frac{2^d}{N^{\frac{d}{(1+2d)}}} +
      \frac{2^{d+3} C'^2d}{N^{\frac{d}{2(1+2d)}}} 
    \right)\\
    &\ll_d 
    N^d 
    \left(  
      \frac{\left| |S| - N^{d/2} \right|}{N^{d/2}} +
      \frac{1}{N^{\frac{d}{2(1+2d)}}} 
    \right).
  \end{align*}
  This concludes \cref{Fourier uniformity}.
\end{proof}

\section{Dense model lemma}
\label{Sec: Dense model lemma}

For any function \(f : [N]^d \to \mathbb{C}\) define 
\[
  \left\| f \right\|^* 
  \coloneq
  \sup_{\substack{g : [N]^d \to \mathbb{C} \\ \left\| \widehat{g} \right\|_\infty \leq 1}}
  \left| 
    \sum_{n \in [N]^d} f(n) g(n)
  \right|.
\]

\begin{lemma} \label{properties of *norm}
  For any \(f, g : [N]^d \to \mathbb{C}\) the following hold.
  \begin{enumerate}
    \item \(\displaystyle 
      \left| \sum_{n \in [N]^d} f(n) g(n)\right| 
      \leq 
      \left\| f \right\|^*
      \left\| \widehat{g} \right\|_\infty\).
    \item \(\displaystyle 
      \left\| f g \right\|^* 
      \leq 
      \left\| f \right\|^* \left\| g \right\|^*\).
    \item \(\displaystyle 
      \left\| f \right\|_\infty 
      \leq 
      \left\| f \right\|^*\).
  \end{enumerate}
\end{lemma}

\begin{proof}
  Let \(f, g : [N]^d \to \mathbb{C}\) be functions. 
  Observe that
  \begin{align*} 
    \left\| f \right\|^* 
    =
    \sup_{\left\| \widehat{h} \right\|_\infty \leq 1}
    \left| \sum_{n \in [N]^d} f(n) h(n)\right|
    \geq 
    \left| \sum_{n \in [N]^d} \frac{f(n) g(n)}{\left\| \widehat{g} \right\|_\infty}\right|.
  \end{align*} 
  This proves the first part.
  Note that for any function \(h : [N]^d \to \mathbb{C}\) and 
  any \(\alpha \in \mathbb{T}^d\) we have
  \begin{align*}
    \left| \widehat{f h} (\alpha) \right|
    &= 
    \left| \sum_{n \in [N]^d} f(n) h(n) e(n \cdot \alpha) \right| \\
    &\leq 
    \left\| f \right\|^* 
    \sup_{\beta \in \mathbb{T}^d}
    \left| \sum_{n \in [N]^d} h(n) e(n \cdot \alpha) e(n \cdot \beta) \right| \\
    &=
    \left\| f \right\|^* \| \widehat{h} \|_\infty.
  \end{align*}
  Therefore, 
  \begin{align*}
    \left\| f g \right\|^* 
    =
    \sup_{\| \widehat{h} \|_\infty \leq 1}
    \left| \sum_{n \in [N]^d} f(n) g(n) h(n) \right|
    \leq
    \sup_{\| \widehat{h} \|_\infty \leq 1}
    \| \widehat{f h} \|_\infty
    \left\| g \right\|^*
    \leq 
    \left\| f \right\|^* \left\| g \right\|^*,
  \end{align*}
  which proves the second part.
  Let \(m\) be any element of \([N]^d\). 
  Note that \(\| \widehat{\mathbf{1}}_{\{m\}} \|_\infty = 1\) and hence
  \begin{align*}
    \left\| f \right\|^* 
    =
    \sup_{\| \widehat{h} \|_\infty \leq 1}
    \left| \sum_{n \in [N]^d} f(n) h(n) \right|
    \geq
    \left| f(m) \right|.
  \end{align*}
  This implies the third part of \cref{properties of *norm}.
\end{proof}

We state the following version of hyperplane separation theorem,
see for instance \cite[Proposition 1.5.3]{Bertsekas_convex-optimization}.

\begin{theorem}[Strict hyperplane separation theorem] \label{Supporting hyperplane theorem}
  Let \(K \subseteq \mathbb{R}^n\) be a nonempty, closed and convex set 
  and \(v \in \mathbb{R}^n \setminus K\). Then there exists \(y \in \mathbb{R}^n\)
  such that \(\langle v, y \rangle > \langle u, y \rangle\) holds for all \(u \in K\),
  where \(\langle \cdot, \cdot \rangle\) denotes 
  the standard inner product on \(\mathbb{R}^n\).
\end{theorem}

Using \cref{properties of *norm,Supporting hyperplane theorem}, 
we establish the following dense model lemma.

\begin{lemma}[Dense model lemma] \label{Dense model lemma}
  Let \(\varepsilon \leq 1\) be a positive real number, \(r\) be a positive integer and 
  \(\nu : [N]^d \to [0, \infty)\) be a function.
  Suppose that \(f_1, \dots, f_r : [N]^d \to [0, \infty)\) are functions with 
  \(f_1 + \dots + f_r \leq \nu\) such that 
  for any \(g_1, \dots, g_r : [N]^d \to [0, \infty)\) satisfying 
  \(g_1 + \dots + g_r \leq \mathbf{1}_{[N]^d}\), there exists \(1 \leq i \leq r\)
  such that 
  \[
    \left\| \widehat{f}_i - \widehat{g}_i \right\|_\infty 
    >
    \varepsilon N^d.
  \]
  Then 
  \[
    \left\| \widehat{\nu} - \widehat{\mathbf{1}}_{[N]^d} \right\|_\infty 
    \gg_{\varepsilon, r}
    N^d
  \]
  holds.
\end{lemma}

\begin{proof}
  Write the set \([N]^d\) as an ordered set 
  \(\left\{ n_1, \dots, n_{N^d} \right\}\).
  For any function \(\psi : [N]^d \to \mathbb{R}\) denote the \(N^d\)-tuple 
  \(\left( \psi(n_1), \dots, \psi(n_{N^d}) \right)\) in \(\mathbb{R}^{N^d}\) by \(t(\psi)\).
  Let \(\mathcal{C}\) denote the collection of all those elements of \(\mathbb{R}^{r N^d}\)
  which can be expressed as
  \[
    \left( t(g_1 + h_1), \dots, t(g_r + h_r) \right)
  \]
  for some functions \(g_1, \dots, g_r : [N]^d \to [0, \infty)\) and
  \(h_1, \dots, h_r : [N]^d \to \mathbb{R}\) satisfying
  \(g_1 + \dots + g_r \leq \mathbf{1}_{[N]^d}\) and 
  \(\| \widehat{h}_i \|_\infty \leq {\varepsilon N^d}/{4}\)
  for all \(1 \leq i \leq r\).
  Note that \(\mathcal{C}\) is closed and convex subset of \(\mathbb{R}^{r N^d}\).
  The element 
  \[
    \left( \frac{2}{2 + \varepsilon} \right) (t(f_1), \dots, t(f_r)) \in \mathbb{R}^{r N^d}
  \]
  does not lie in \(\mathcal{C}\), 
  because if 
  \(
    \frac{2}{2 + \varepsilon} f_i = g_i + h_i 
  \)
  for all \(1 \leq i \leq r\)
  and \(g_1 + \dots + g_r \leq \mathbf{1}_{[N]^d}\),
  then our assumption in the statement implies that 
  there exists \(1 \leq i \leq r\) such that
  \(\| \widehat{h}_i \|_\infty > {\varepsilon N^d}/{4}\).
  Therefore, using \cref{Supporting hyperplane theorem} we obtain that there exist
  functions \(\phi_1, \dots, \phi_r : [N]^d \to \mathbb{R}\) such that 
  for any \(g_1, \dots, g_r : [N]^d \to [0, \infty)\) 
  with \(g_1 + \dots + g_r \leq \mathbf{1}_{[N]^d}\) and 
  any \(h_1, \dots, h_r : [N]^d \to \mathbb{R}\) with 
  \(\| \widehat{h}_i \|_\infty \leq \varepsilon N^d / 4\) 
  for all \(1 \leq i \leq r\), we have
  \begin{align} \label{existence of phi}
    \frac{2}{2 + \varepsilon}
    \sum_{i=1}^{r} \sum_{n \in [N]^d}
    f_i(n) \phi_i(n) 
    > 
    \sum_{i=1}^{r} \sum_{n \in [N]^d}
    \left( g_i(n) + h_i(n) \right) \phi_i(n).
  \end{align}
  In particular, taking \(g_i = 0 = h_i\) for all \(1 \leq i \leq r\) implies that
  the quantity in the left-hand side is positive.
  Therefore, by scaling the functions
  \(\phi_1, \dots, \phi_r\) we may assume that 
  \begin{equation*} 
    \sum_{i=1}^{r} \sum_{n \in [N]^d}
    f_i(n) \phi_i(n) 
    = 
    \left( 1 + \frac{\varepsilon}{2} \right) N^d,
  \end{equation*}
  and hence, using \eqref{existence of phi} it follows that
  \begin{equation} \label{normalized phi}
    \sum_{i=1}^{r} \sum_{n \in [N]^d}
    g_i(n) \phi_i(n)
    +
    \sum_{i=1}^{r} \sum_{n \in [N]^d}
    h_i(n) \phi_i(n)
    < 
    N^d.
  \end{equation}
  Consider the function \(\phi : [N]^d \to [0, \infty)\) defined as 
  \[
    \phi (n) 
    =
    \max \left\{ 0, \phi_1(n), \dots, \phi_r(n) \right\},
  \]
  and note that 
  \begin{equation} \label{bound for sum nu}
    \left( 1 + \frac{\varepsilon}{2} \right) N^d
    =
    \sum_{i=1}^{r} \sum_{n \in [N]^d}
    f_i(n) \phi_i(n) 
    \leq 
    \sum_{n \in [N]^d}
    \nu(n) \phi(n).
  \end{equation}
  For each \(n \in [N]^d\) fix \(1 \leq i_n \leq r\) such that 
  \(\phi_{i_n}(n) = \max \left\{ \phi_1(n), \dots, \phi_r(n) \right\}\).
  For every \(1 \leq i \leq r\), taking 
  \[
    g_i(n) =
    \begin{cases}
      1 &\text{if } i = i_n \text{ and } \phi_{i_n}(n) \geq 0 \\
      0 &\text{otherwise},
    \end{cases}
  \]
  and \(h_i = 0\),
  and using \eqref{normalized phi} we deduce that
  \begin{equation} \label{sum of phi}
    \sum_{n \in [N]^d} \phi(n)
    =
    \sum_{i=1}^{r} \sum_{n \in [N]^d}
    g_i(n) \phi_i(n)
    <
    N^d.
  \end{equation}
  Combining \eqref{bound for sum nu} and \eqref{sum of phi} we obtain that
  \begin{equation} \label{diff of nu and 1}
    \sum_{n \in [N]^d}
    \left( \nu(n) - 1 \right) \phi(n)
    >
    \frac{\varepsilon N^d}{2}.
  \end{equation}

  Given any \(1 \leq i \leq r\), taking \(h_j = 0\) for all \(j \neq i\) and
  \(g_j = 0\) for all \(1 \leq j \leq r\) in \eqref{normalized phi},
  we obtain that 
  \begin{equation*}
    \sup_{\| \widehat{h} \|_\infty \leq \frac{\varepsilon N^d}{4}}
    \left|  
      \sum_{n \in [N]^d}
      h(n) \phi_i(n)
    \right|
    \leq
    N^d.
  \end{equation*}
  Therefore, \cref{properties of *norm} implies that
  \begin{equation} \label{sup of phi}
    \left\| \phi_i \right\|_\infty 
    \leq \left\| \phi_i \right\|^*
    \leq \frac{4}{\varepsilon}
  \end{equation}
  holds for all \(1 \leq i \leq r\).
  By the Stone--Weierstrass theorem, there exists a polynomial \(P\)
  in \(r\) variables,
  depending only on \(\varepsilon, r\) such that for any 
  \((y, \dots, y_r) \in \mathbb{R}^r\) with \(|y_i| \leq 4/\varepsilon\)
  for all \( 1 \leq i \leq r\) we have 
  \[
    \left|  
      \max \left\{0, y_1, \dots, y_r \right\} - P(y_1, \dots, y_r)
    \right|
    \leq 
    \frac{\varepsilon}{12}.
  \]
  Note that if \(\sum_{n \in [N]^d} \nu(n) \geq 2N^d\) then 
  \(\| \widehat{\nu} - \widehat{\mathbf{1}}_{[N]^d} \|_\infty \geq N^d\)
  and we are done.
  Assume that \(\sum_{n \in [N]^d} \nu(n) \leq 2N^d\) holds.
  Thus, we obtain that
  \begin{align*}
    &\sum_{n \in [N]^d}
    \left( \nu(n) - 1 \right) \phi(n) \\
    &\leq 
    \sum_{n \in [N]^d}
    \nu(n) \left(P \left( \phi_1(n), \dots, \phi_r(n) \right) + \frac{\varepsilon}{12}\right)
    -
    \sum_{n \in [N]^d}
    \left(P \left( \phi_1(n), \dots, \phi_r(n) \right) - \frac{\varepsilon}{12} \right) \\
    &\leq 
    \sum_{n \in [N]^d}
    \left( \nu(n) - 1 \right)
    P \left( \phi_1(n), \dots, \phi_r(n) \right)
    + 
    \frac{\varepsilon N^d}{4}.
  \end{align*}
  Using \eqref{diff of nu and 1} and the pigeonhole principle,
  we obtain that there exist a nonnegative integer \(R\) and a positive 
  constant \(C''\), both depending on \(\varepsilon, r\) such that 
  \[
    \sum_{n \in [N]^d}
    \left( \nu(n) - 1 \right)
    \psi_1(n) \cdots \psi_R(n)
    \geq 
    \frac{\varepsilon N^d}{C''}
  \]
  where \(\psi_i \in \left\{\phi_1, \dots, \phi_r \right\}\)
  for all \(0 \leq i \leq R\).
  Using \cref{properties of *norm} and \eqref{sup of phi}
  we obtain that 
  \begin{align*}
    \sum_{n \in [N]^d}
    \left( \nu(n) - 1 \right)
    \psi_1(n) \cdots \psi_R(n)
    \leq 
    \left\| \widehat{\nu} - \widehat{\mathbf{1}}_{[N]^d} \right\|_\infty 
    \left\| \psi_1 \cdots \psi_R \right\|^*
    \ll_{\varepsilon,r}
    \left\| \widehat{\nu} - \widehat{\mathbf{1}}_{[N]^d} \right\|_\infty.
  \end{align*}
  This concludes the proof.
\end{proof}

\section{Partition regularity over Sidon sets}
\label{Sec: Partition regularity}

Using \cref{Robertson higher dimensional Rado} for \(t = 1\),
we prove \cref{Counting monochromatic solutions} which gives 
a lower bound for the weighted count of the solutions to
a partition regular linear equation for a certain weight function.

\begin{lemma}[Counting weighted monochromatic solutions]
\label{Counting monochromatic solutions}
  Let \(s\) be a positive integer and \(a_1, \dots, a_s\) be nonzero integers with 
  \(\sum_{i \in I} a_i = 0\) for some nonempty subset \(I \subseteq [s]\).
  Then for every positive integer \(r\) there exist positive constants 
  \(c, N_0\),
  both depending on \(d, r, s, a_1, \dots, a_s\),
  such that for every \(N \geq N_0\) the following holds.
  For any \(g_1, \dots, g_r : [N]^d \to [0, \infty)\) with 
  \(g_1 + \dots + g_r = \mathbf{1}_{[N]^d}\) there exists \(1 \leq i \leq r\)
  such that 
  \[
    \sum_{a_1 x_1 + \dots + a_s x_s = 0}
    g_i(x_1) \cdots g_i(x_s) 
    \geq 
    c N^{d(s-1)}.
  \]
\end{lemma}

\begin{proof}
  Let \(r\) be a positive integer and \(g_1, \dots, g_r : [N]^d \to [0, \infty)\)
  be functions satisfying \(g_1 + \dots + g_r = \mathbf{1}_{[N]^d}\).
  Using the pigeonhole principle, for every \(n \in [N]^d\) 
  fix an integer \(1 \leq i_n \leq r\) such that 
  \(g_{i_n}(n) \geq 1/r\). 
  For every \(1 \leq i \leq r\) define the set 
  \[
    C_i
    \coloneq
    \left\{  
      n \in [N]^d :
      i_n = i
    \right\}
  \]
  and note that 
  \[
    [N]^d = \bigsqcup_{i=1}^{r} C_i.
  \]
  Using the pigeonhole principle,
  \cref{Robertson higher dimensional Rado} implies that there exist
  an integer \(1 \leq i \leq r\), and positive constants 
  \(c_1, N_0\), 
  both depending on \(d, r, s, a_1, \dots, a_s\),
  such that the number of solutions to the equation 
  \(a_1 x_1 + \dots + a_s x_s = 0\) in \(C_i\) is at least \(c_1 N^{d(s-1)}\).
  Therefore, we obtain that 
  \begin{align*}
    \sum_{a_1 x_1 + \dots + a_s x_s = 0}
    g_i(x_1) \cdots g_i(x_s) 
    &\geq
    \sum_{\substack{a_1 x_1 + \dots + a_s x_s = 0 \\ x_1, \dots, x_s \in C_i}}
    g_i(x_1) \cdots g_i(x_s) \\
    &\geq
    \frac{1}{r^s}
    \sum_{\substack{a_1 x_1 + \dots + a_s x_s = 0}}
    \mathbf{1}_{C_i}(x_1) \cdots \mathbf{1}_{C_i}(x_s) \\
    &\geq 
    \frac{c_1 N^{d(s-1)}}{r^s}.
  \end{align*}
  This concludes the result.
\end{proof}

Now we prove \cref{Partition regularity}, which
shows that the partition regular linear equations admit 
many monochromatic solutions, for any given partition of a sufficiently
dense Sidon subset of \([N]^d\), for all large \(N\).
We prove it using 
\cref{Fourier uniformity,Dense model lemma,Counting monochromatic solutions}.

\begin{proof} [Proof of \cref{Partition regularity}]
  Let \(r\) be a positive integer.
  Consider positive constants \(c < 1\) and \(N_0\) 
  to be specified later.
  Let \(N \geq N_0\) be an integer and \(S \subseteq [N]^d\) be a Sidon set
  with \(N^{d/2}/2 \leq |S| \leq N^{d/2}\).
  Assume that there exists an \(r\)-coloring of \(S\), 
  say \(S = C_1 \sqcup \dots \sqcup C_r\) such that
  \[
    \sum_{a_1 x_1 + \dots + a_s x_s = 0}
    \mathbf{1}_{C_i}(x_1) \cdots \mathbf{1}_{C_i}(x_s) 
    \leq 
    c |S|^s N^{-d}
  \]
  holds for all \(1 \leq i \leq r\).
  For every \(1 \leq i \leq r\) define 
  \[
    f_i \coloneq
    \frac{N^d}{|S|} \mathbf{1}_{C_i}
  \]
  and note that 
  \(f_1 + \dots + f_r = \frac{N^d}{|S|} \mathbf{1}_{S}\)
  holds.
  Let \(g_1, \dots, g_r : [N]^d \to [0, \infty)\) be functions satisfying 
  \(g_1 + \dots + g_r = \mathbf{1}_{[N]^d}\).
  \cref{Counting monochromatic solutions} implies that there exist
  positive constants \(c_1, N_1\),
  both depending on \(d, r, s, a_1, \dots, a_s\),
  such that 
  taking \(N_0 \geq N_1\) and \(c = c_1/2\) yields that 
  there exists \(1 \leq i \leq r\) such that 
  \begin{align*}
    &\sum_{a_1 x_1 + \dots + a_s x_s = 0}
    \left(  
      g_i(x_1) \cdots g_i(x_s) - f_i(x_1) \cdots f_i(x_s) 
    \right) \\
    &\geq 
    c_1 N^{d(s-1)} -
    \left( \frac{N^d}{|S|} \right)^s
    \sum_{a_1 x_1 + \dots + a_s x_s = 0}
    \mathbf{1}_{C_i}(x_1) \cdots \mathbf{1}_{C_i}(x_s) \\
    &\geq
    c N^{d(s-1)}.
  \end{align*}
  Using a telescoping identity, the orthogonality, and 
  the H\"older's inequality,
  we obtain that 
  \begin{multline} \label{estimate of f-g}
    \left|  
      \sum_{a_1 x_1 + \dots + a_s x_s = 0}
      \left(  
        g_i(x_1) \cdots g_i(x_s) - f_i(x_1) \cdots f_i(x_s) 
      \right)
    \right| \\
    \leq
    s \| \widehat{f}_i - \widehat{g}_i \|_\infty
    \max 
    \left\{ 
      \| \widehat{f}_i \|_{s-1}^{s-1}, \| \widehat{g}_i \|_{s-1}^{s-1} 
    \right\}.
  \end{multline}
  Using the assumption \(|S| \geq N^{d/2}/2\) observe that
  \begin{align*}
    \| \widehat{f}_i \|_{s-1}^{s-1} 
    &=
    \int_{\mathbb{T}^d}
    \left| \sum_{n \in [N]^d} f_i(n) e(n \cdot \theta) \right|^{s-1} 
    \mathrm{d}\theta \\
    &\leq
    \frac{N^{d(s-1)}}{|S|^4} 
    \sum_{n_1 - n_2 = m_1 - m_2} 
    \mathbf{1}_S(n_1) \mathbf{1}_S(n_2) \mathbf{1}_S(m_1) \mathbf{1}_S(m_2) \\
    &\leq 
    8 N^{d(s-2)}.
  \end{align*}
  The Parseval's identity implies that
  \begin{align*}
    \| \widehat{g}_i \|_{s-1}^{s-1} 
    =
    \int_{\mathbb{T}^d}
    \left| \sum_{n \in [N]^d} g_i(n) e(n \cdot \theta) \right|^{s-1} 
    \mathrm{d}\theta
    \leq
    N^{d(s-3)}
    \| g_i \|_2^2
    \leq 
    N^{d(s-2)}.
  \end{align*}
  Using \eqref{estimate of f-g} it follows that 
  \begin{equation} \label{large Fourier uniform norm of f-g}
    \| \widehat{f}_i - \widehat{g}_i \|_\infty
    \geq
    \frac{c}{8s}
    N^d.
  \end{equation}

  Note that if 
  \[
    \| \widehat{f}_j - \widehat{g}_j \|_\infty
    \leq
    \frac{c}{16rs}
    N^d
  \]
  for all \(j \neq i\), then the triangle inequality implies that
  \begin{align*}
    \left\| 
      \frac{N^d}{|S|} \widehat{\mathbf{1}}_{S} - \widehat{\mathbf{1}}_{[N]^d} 
    \right\|_\infty
    \geq 
    \| \widehat{f}_i - \widehat{g}_i \|_\infty - 
    \sum_{j \neq i} \| \widehat{f}_j - \widehat{g}_j \|_\infty
    \geq 
    \frac{c}{16s} N^d.
  \end{align*}
  On the other hand, if there exists \(j \in \left\{ 1, \dots, r \right\} \setminus \{i\}\)
  such that
  \[
    \| \widehat{f}_j - \widehat{g}_j \|_\infty
    >
    \frac{c}{16rs}
    N^d,
  \]
  then \cref{Dense model lemma} implies that 
  \[
    \left\| 
      \frac{N^d}{|S|} \widehat{\mathbf{1}}_{S} - \widehat{\mathbf{1}}_{[N]^d} 
    \right\|_\infty
    \gg_{c, r, s}
    N^d.
  \]
  Now using \cref{Fourier uniformity} we deduce that
  \begin{align*}
    N^{d/2}
    \left(  
      \left| \frac{|S|}{N^{d/2}} - 1 \right| + 
      \frac{1}{N^{\frac{d}{2(1+2d)}}}
    \right)^{1/2}
    \gg_{d}
    \left\| 
      \widehat{\mathbf{1}}_S -
      \frac{|S|}{N^d} \widehat{\mathbf{1}}_{[N]^d}
    \right\|_\infty
    \gg_{c, r, s}
    |S|.
  \end{align*}
  Thus, there exists a positive constant 
  \(c_2 < 1\) depending on \(d, r, s, a_1, \dots, a_s\) such that
  \begin{align*}
    N^{d/2} - |S| 
    \geq
    c_2 N^{d/2} - N^{{d^2}/{(1+2d)}}.
  \end{align*}
  Hence, taking 
  \(N_0 = \max \left\{ N_1, \left\lceil 2/c_2 \right\rceil^6 \right\}\) 
  yields that
  \[
    |S| \leq 
    \left( 1 - \frac{c_2}{2} \right) N^{d/2}.
  \]
  This proves \cref{Partition regularity} for \(|S| \leq N^{d/2}\).
  This implies that \cref{Partition regularity} holds for  
  \(|S| \geq N^{d/2}\).
  Hence, the result is proved.
\end{proof}

\section{Equidistribution of dense Sidon sets}
\label{Sec: Equidistribution of dense Sidon sets}

In \cite[Lemma~3.2]{OrtegaPrendiville_extremal-Sidon-sets},
Ortega and Prendiville proved that every nontrivial finite arithmetic 
progression \(P\) in \(\mathbb{Z}\) satisfies
\begin{equation} \label{OP lemma 3.2}
  \int_{\mathbb{T}} 
    \left|  
      \widehat{\mathbf{1}}_P (\alpha) 
    \right|
    \mathrm{d} \alpha
    \ll
    \log |P|.
\end{equation}

The following result follows using \eqref{OP lemma 3.2}.

\begin{lemma} \label{lower bound for Fourier L1 norm}
  Let \(P \coloneq P_1 \times \dots \times P_d \subseteq \mathbb{Z}^d\),
  where \(P_i\) is a nontrivial arithmetic progression in \(\mathbb{Z}\)
  for all \(1 \leq i \leq d\). Then 
  \[
    \int_{\mathbb{T}^d} 
    \left|  
      \widehat{\mathbf{1}}_P (\alpha) 
    \right|
    \mathrm{d} \alpha
    \ll
    \prod_{i=1}^{d}
    \log |P_i|.
  \]
\end{lemma}

Using the above result and \cref{Fourier uniformity}
we prove the following.

\begin{lemma} 
\label{uniformity of Sidon in progression}
  Let \(S \subseteq [N]^d\) be a Sidon set, 
  \(P \coloneq P_1 \times \dots \times P_d \subseteq \mathbb{Z}^d\)
  where \(P_i\) is an arithmetic progression in \([N]\)
  for all \(1 \leq i \leq d\).
  Then 
  \[
    |S \cap P| - \frac{|S| |P|}{N^d}
    \ll_{d, \varepsilon}
    N^{\varepsilon + d/2}
    \left(  
      \left| \frac{|S|}{N^{d/2}} - 1 \right| + 
      {N^{-\frac{d}{2(1+2d)}}}
    \right)^{1/2}
  \]
  holds for every \(\varepsilon > 0\).
\end{lemma}

\begin{proof}
  Using the orthogonality we observe that 
  \begin{align*}
    &\int_{\mathbb{T}^d}
    \widehat{\mathbf{1}}_{P} (\alpha)
    \left(  
      \widehat{\mathbf{1}}_S  -
      \frac{|S|}{N^d} \widehat{\mathbf{1}}_{[N]^d}
    \right) (-\alpha)
    \mathrm{d}\alpha \\
    &=
    \int_{\mathbb{T}^d}
    \sum_{n \in P} \sum_{m \in S} 
    e \left( (n - m) \cdot \alpha \right) 
    \mathrm{d} \alpha
    -
    \frac{|S|}{N^d}
    \int_{\mathbb{T}^d}
    \sum_{n \in P} \sum_{m \in [N]^d} 
    e \left( (n - m) \cdot \alpha \right) 
    \mathrm{d} \alpha \\
    &=
    |S \cap P| - \frac{|S| |P|}{N^d}.
  \end{align*}
  Therefore, 
  \cref{Fourier uniformity,lower bound for Fourier L1 norm} imply that
  given any \(\varepsilon > 0\) we have
  \begin{align*}
    \left|  
      \int_{\mathbb{T}^d}
      \widehat{\mathbf{1}}_{P} (\alpha)
      \left( 
        \widehat{\mathbf{1}}_S  -
        \frac{|S|}{N^d} \widehat{\mathbf{1}}_{[N]^d}
      \right) (-\alpha)
      \mathrm{d}\alpha
    \right| 
    &\leq
    \left\| 
      \widehat{\mathbf{1}}_S -
      \frac{|S|}{N^d} \widehat{\mathbf{1}}_{[N]^d}
    \right\|_\infty
    \int_{\mathbb{T}^d} 
    \left|  
      \widehat{\mathbf{1}}_P (\alpha) 
    \right|
    \mathrm{d} \alpha\\
    &\ll_{d, \varepsilon}
    N^{\varepsilon + d/2}
    \left(  
      \left| \frac{|S|}{N^{d/2}} - 1 \right| + 
      {N^{-\frac{d}{2(1+2d)}}}
    \right)^{1/2}.
  \end{align*}
  This proves \cref{uniformity of Sidon in progression}.
\end{proof}

\cref{uniformity of Sidon in progression} implies 
\cref{Equidistribution in intervals,Equidistribution in residue classes},
which show that dense Sidon sets are equidistributed in certain 
structures.

\begin{corollary} 
\label{Equidistribution in intervals}
  Let \(I \coloneq I_1 \times \dots \times I_d\) be a rectangle in \([0,1]^d\) 
  and \(S \subseteq [N]^d\) be a Sidon set.
  If \(|S| \geq \delta N^{d/2}\) for some \(\delta > 0\) then
  \[
    \E_{n \in S} \mathbf{1}_I \left( n/N \right)
    =
    \mathrm{meas} (I) +
    O_{d, \varepsilon, \delta}
    \left(  
      N^{\varepsilon}
      \left(  
        \left| \frac{|S|}{N^{d/2}} - 1 \right| + 
        {N^{-\frac{d}{2(1+2d)}}}
      \right)^{1/2}
    \right)
  \]
  holds for every \(\varepsilon > 0\).
\end{corollary}
The above result follows by applying \cref{uniformity of Sidon in progression}
to
\(
  P = 
  \left\{ 
    n \in [N]^d : n \in N I
  \right\}
\) 
and observing 
\[
  \prod_{j=1}^{d} (\mathrm{meas}(N I_j) - 1)
  \leq
  |P| 
  \leq 
  \prod_{j=1}^{d} (\mathrm{meas}(N I_j) + 1).
\]

\begin{corollary} 
\label{Equidistribution in residue classes}
  Let \(q_1, \dots, q_d\) be positive integers, 
  \(r \coloneq (r_1, \dots, r_d) \in \mathbb{Z}^d\)
  and \(S \subseteq [N]^d\) be a Sidon set.
  Let \(P\) denote the set \(P_1 \times \dots \times P_d\) 
  where \(P_i \coloneq \left\{ n q_i + r_i : n \in \mathbb{Z} \right\}\).
  If \(|S| \geq \delta N^{d/2}\) for some \(\delta > 0\) then
  \[
    \E_{n \in S} \mathbf{1}_{P} (n)
    =
    \frac{1}{q_1 \cdots q_d} +
    O_{d, \varepsilon, \delta}
    \left(  
      N^{\varepsilon}
      \left(  
        \left| \frac{|S|}{N^{d/2}} - 1 \right| + 
        {N^{-\frac{d}{2(1+2d)}}}
      \right)^{1/2}
    \right)
  \]
  holds for all \(\varepsilon > 0\).
\end{corollary}

\cref{Equidistribution in residue classes} follows by applying
\cref{uniformity of Sidon in progression} and using the relation
\begin{align*}
  \prod_{i=1}^{d}\left( \frac{N}{q_i} - 1 \right)
  \leq
  |P \cap [N]^d|
  \leq 
  \prod_{i=1}^{d}\left( \frac{N}{q_i} + 1 \right).
\end{align*}
Now we show the equidistribution of dense Sidon subsets of \([N]^d\)
in regular Bohr subsets, by using the arguments of 
\cite{OrtegaPrendiville_extremal-Sidon-sets}.
A function \(F : \mathbb{T}^k \to \mathbb{R}\) is said to be 
\emph{\(M\)-Lipschitz}
if given any \(\theta, \theta' \in \mathbb{T}^k\) we have 
\[
  \left| F(\theta) - F(\theta') \right|
  \leq 
  M \left\| \theta - \theta' \right\|.
\]

\begin{lemma} \cite[Lemma~4.1]{OrtegaPrendiville_extremal-Sidon-sets}
\label{Trigonometric approximation}
  Let \(F : \mathbb{T}^{k} \to [0,1]\) be an \(M\)-Lipschitz function
  for some \(M \geq 1\).
  Then for every \(\varepsilon > 0\) there exists a function
  \(F_\varepsilon : \mathbb{T}^{k} \to [0,1]\) with 
  \(
    \left\| F - F_\varepsilon \right\|_\infty 
    \leq
    \varepsilon
  \)
  and satisfying 
  \[
    F_\varepsilon(\theta)
    =
    \sum_{y \in [-Y,Y]^k}
    \widehat{F}_\varepsilon(y) e(\theta \cdot y)
  \]
  for some \(Y \ll M^2 / \varepsilon^3\).
\end{lemma}

A proof of the above result can be found in 
\cite[Appendix B]{OrtegaPrendiville_extremal-Sidon-sets}.
For elements $n \coloneq (n_1, \dots, n_d) \in \mathbb{Z}^d$ and 
$\alpha \coloneq (\alpha_1, \dots, \alpha_k) \in \mathbb{T}^{dk}$
where \(\alpha_i \in \mathbb{T}^d\) for all \(1 \leq i \leq k\) 
we write
\[
  n \odot \alpha \coloneq (n \cdot \alpha_1, \dots, n \cdot \alpha_k).
\]

\begin{lemma} 
\label{Equidistribution for trigonometric functions}
  Let \(F : \mathbb{T}^{k} \to [0,1]\) be an \(M\)-Lipschitz function
  for some \(M \geq 1\).
  Let \(S \subseteq [N]^d\) be a Sidon set with 
  \(|S| \geq \delta N^{d/2}\) for some \(\delta > 0\).
  Then for every \(\alpha \in \mathbb{T}^{dk}\) we have
  \[
    \E_{n \in S} F (n \odot \alpha)
    =
    \E_{n \in [N]^d} F (n \odot \alpha) +
    O_{d, k, \delta}
    \left(  
      M^{2/3}
      \left(  
        \left| \frac{|S|}{N^{d/2}} - 1 \right| + 
        {N^{-\frac{d}{2(1+2d)}}}
      \right)^{1/8k}
    \right).
  \]
\end{lemma}

\begin{proof}
  Let \(\varepsilon > 0\) be a real number to be determined,
  and \(\alpha \in \mathbb{T}^{dk}\).
  Applying \cref{Trigonometric approximation} we obtain a 
  function \(F_\varepsilon : \mathbb{T}^k \to [0,1]\) with 
  \(
    \left\| F - F_\varepsilon \right\|_\infty 
    \leq
    \varepsilon
  \)
  and satisfying 
  \[
    F_\varepsilon(\theta)
    =
    \sum_{y \in [-Y,Y]^k}
    \widehat{F}_\varepsilon(y) e(\theta \cdot y)
  \]
  for some \(Y \ll M^2 / \varepsilon^3\).
  It follows that
  \begin{align*}
    &\left|  
      \sum_{n \in S} F_\varepsilon (n \odot \alpha)
      -
      \frac{|S|}{N^d}
      \sum_{n \in [N]^d} F_\varepsilon (n \odot \alpha)
    \right| \\
    &=
    \left|  
      \sum_{y \in [-Y,Y]^k}
      \widehat{F}_\varepsilon(y) 
      \left(  
        \sum_{n \in S}
        e( (n \odot \alpha) \cdot y) - 
        \frac{|S|}{N^d}
        \sum_{n \in [N]^d}
        e( (n \odot \alpha) \cdot y)
      \right) 
    \right| \\
    &\ll_k
    Y^k 
    \left\| 
      \widehat{\mathbf{1}}_S -
      \frac{|S|}{N^d} \widehat{\mathbf{1}}_{[N]^d}
    \right\|_\infty.
  \end{align*}
  Note that 
  \[
    \left|  
      \E_{n \in S}
      \left( F (n \odot \alpha) - F_\varepsilon (n \odot \alpha) \right)
    \right|
    \leq
    \varepsilon
  \]
  and 
  \[
    \left|  
      \E_{n \in [N]^d}
      \left( F (n \odot \alpha) - F_\varepsilon (n \odot \alpha) \right)
    \right|
    \leq
    \varepsilon.
  \]
  Therefore, using the triangle inequality and the relation \(Y \ll M^2 / \varepsilon^3\),
  we deduce that
  \begin{align*}
    \left|  
      \E_{n \in S} F (n \odot \alpha)
      -
      \E_{n \in [N]^d} F (n \odot \alpha)
    \right|
    \ll_k
    \varepsilon +
    \frac{M^{2k}}{\varepsilon^{3k} |S|} 
    \left\| 
      \widehat{\mathbf{1}}_S -
      \frac{|S|}{N^d} \widehat{\mathbf{1}}_{[N]^d}
    \right\|_\infty.
  \end{align*}
  Taking \(\varepsilon > 0\) satisfying
  \[
    \varepsilon ^{3k+1}
    =
    \frac{M^{2k}}{|S|}
    \left\| 
      \widehat{\mathbf{1}}_S -
      \frac{|S|}{N^d} \widehat{\mathbf{1}}_{[N]^d}
    \right\|_\infty
  \]
  and using \cref{Fourier uniformity} we obtain that
  \begin{align*}
    \left|  
      \E_{n \in S} F (n \odot \alpha)
      -
      \E_{n \in [N]^d} F (n \odot \alpha)
    \right|
    &\ll_k
    \left(  
      \frac{M^{2k}}{|S|}
      \left\| 
        \widehat{\mathbf{1}}_S -
        \frac{|S|}{N^d} \widehat{\mathbf{1}}_{[N]^d}
      \right\|_\infty
    \right)^{1/(3k + 1)} \\
    &\ll_{d, k, \delta}
    \left(  
      M^{2k}
      \left(  
        \left| \frac{|S|}{N^{d/2}} - 1 \right| + 
        {N^{-\frac{d}{2(1+2d)}}}
      \right)^{1/2}
    \right)^{1/(3k + 1)}.
  \end{align*}
  Note that we have \(M \geq 1\) and 
  \[
    \left| \frac{|S|}{N^{d/2}} - 1 \right| < 1
  \]
  for all large \(N\).
  Therefore, we deduce that 
  \[
    \left|  
      \E_{n \in S} F (n \odot \alpha)
      -
      \E_{n \in [N]^d} F (n \odot \alpha)
    \right|
    \ll_{d, k, \delta} 
    M^{2/3}
      \left(  
        \left| \frac{|S|}{N^{d/2}} - 1 \right| + 
        {N^{-\frac{d}{2(1+2d)}}}
      \right)^{1/8k}.
  \]
  This proves \cref{Equidistribution for trigonometric functions}.
\end{proof}
For an element \(\alpha \in \mathbb{T}^{dk}\)
and a positive real number \(\eta < 1\), the Bohr set
\[
  B_{\eta, \alpha}(N)
  \coloneq
  \left\{  
    n \in [N]^d 
    :
    \left\| n \odot \alpha \right\| 
    \leq 
    \eta 
  \right\}
\]
is said to be \emph{regular} if for every \(0 < \kappa < 1\) the relation
\begin{align*}
  1 - \kappa
  \leq 
  \frac{|B_{\eta_1, \alpha}(N_1)|}{|B_{\eta, \alpha}(N)|}
  \leq
  1 + \kappa
\end{align*}
holds for all \(\eta_1\) and \(N_1\) satisfying 
\(|\eta - \eta_1| < \kappa \eta / 100k\) and
\(|N - N_1| < \kappa N / 100k\).

\begin{lemma} \label{Higher-dimensional regular Bohr sets}
  Given any element \(\alpha \in \mathbb{T}^{dk}\) and
  positive real numbers \(\eta < 1\) and \(N\), 
  there exist real numbers \(\eta_1 \in (\eta / 2, \eta)\) and 
  \(N_1 \in (N/2, N)\) such that 
  the Bohr set \(B_{\eta_1, \alpha}(N_1)\) is regular.  
\end{lemma}

\cref{Higher-dimensional regular Bohr sets} can be proved using the arguments of 
\cite[Lemma~3.0]{Bourgain_Bohr_set},
where 
Bourgain proved the above result for \(k = 1\).

\begin{corollary} \label{Equidistribution in regular Bohr sets}
  Let \(B \coloneq B_{\eta, \alpha}(N)\) be a regular Bohr set
  for some \(\alpha \in \mathbb{T}^{dk}\) and \(0 < \eta \leq 1/2\),
  and \(S \subseteq [N]^d\) be a Sidon set with 
  \(|S| \geq \delta N^{d/2}\) for some \(\delta > 0\).
  Then we have 
  \[
    \E_{n \in S} \mathbf{1}_B(n)
    =
    \E_{n \in [N]^d} \mathbf{1}_B(n) +
    O_{d, k, \delta}
    \left(  
      \frac{1}{\eta^{2/3}}
      \left(  
        \left| \frac{|S|}{N^{d/2}} - 1 \right| + 
        {N^{-\frac{d}{2(1+2d)}}}
      \right)^{1/14k}
    \right).
  \]
\end{corollary}

\begin{proof}
  Note that if \(\eta = 1/2\) then \(B = [N]^d\), and hence, the result holds.
  Assume that \(\eta < 1/2\).
  Let \(\varepsilon \leq 1\)
  be a positive real number to be determined later. 
  Define the functions \(G_1, G_2 : [0,1] \to \mathbb{R}\) as 
  \[
    G_1(\theta)
    =
    \begin{cases}   
      \min 
      \left\{  1,
        \frac{\eta - \theta}{\varepsilon}
      \right\}
      &\text{ if } \theta \in [0, \eta] \\  
      0 &\text{ if } \theta \in (\eta, 1 - \eta) \\       
      \min 
      \left\{  1,
        \frac{\theta + \eta - 1}{\varepsilon}
      \right\}
      &\text{ if } \theta \in [1 - \eta, 1]
    \end{cases}
  \]
  and
  \[
    G_2(\theta)
    =
    \begin{cases} 
      1 &\text{ if } \theta \notin (\eta, 1 - \eta) \\      
      \max 
      \left\{  0,
        \frac{\eta + \varepsilon - \theta}{\varepsilon} ,
        \frac{\theta + \eta + \varepsilon - 1}{\varepsilon} 
      \right\}
      &\text{ if } \theta \in (\eta, 1 - \eta).
    \end{cases}
  \]
  Note that the functions \(G_1\) and \(G_2\) can be considered as functions
  on \(\mathbb{T}\).
  For every \(j \in \left\{ 1,2 \right\}\) 
  define a function \(F_j : \mathbb{T}^k \to \mathbb{R}\) such that 
  for any \(\theta \coloneq (\theta_1, \dots, \theta_k) \in \mathbb{T}^k\) we have 
  \[
    F_j(\theta) = 
    \prod_{i=1}^{k} G_j(\theta_i).
  \]
  Using a telescoping identity, observe that the functions 
  \(F_1, F_2\) are \(k/\varepsilon\)-Lipschitz, and 
  \[
    F_1(n \odot \alpha) \leq \mathbf{1}_B(n) \leq F_2(n \odot \alpha)
  \] 
  holds for all \(n \in [N]^d\).
  Therefore,
  \cref{Equidistribution for trigonometric functions} implies that there exists
  a positive constant \(C\) depending on \(d,k,\delta\) such that 
  \begin{multline*}
    \E_{n \in [N]^d} F_1(n \odot \alpha) -
    C 
    \left( \frac{k}{\varepsilon} \right)^{2/3}
    \left(  
      \left| \frac{|S|}{N^{d/2}} - 1 \right| + 
      {N^{-\frac{d}{2(1+2d)}}}
    \right)^{1/8k}
    \leq
    \E_{n \in S} \mathbf{1}_B(n) \\
    \leq
    \E_{n \in [N]^d} F_2(n \odot \alpha) +
    C 
    \left( \frac{k}{\varepsilon} \right)^{2/3}
    \left(  
      \left| \frac{|S|}{N^{d/2}} - 1 \right| + 
      {N^{-\frac{d}{2(1+2d)}}}
    \right)^{1/8k}
  \end{multline*}
  holds.
  Assume that \(\varepsilon < \eta/200 k\). 
  The regularity property of \(B\) implies that
  \begin{align*} 
    \left| 
      \E_{n \in [N]^d} F_j (n \odot \alpha)
      -
      \E_{n \in [N]^d} \mathbf{1}_B(n)
    \right|
    \leq
    \frac{200k \varepsilon}{\eta}
  \end{align*}
  holds for \(j = 1,2\).
  Hence, we obtain that 
  \begin{align*}
    \E_{n \in S} \mathbf{1}_B(n)
    -
    \E_{n \in [N]^d} \mathbf{1}_B(n) 
    =
    O_{d, k, \delta}
    \left(  
      \frac{\varepsilon}{\eta} +
      \frac{1}{\varepsilon^{2/3}}
      \left(  
        \left| \frac{|S|}{N^{d/2}} - 1 \right| + 
        {N^{-\frac{d}{2(1+2d)}}}
      \right)^{1/8k}
    \right).
  \end{align*}
   Using \cref{size of Sidon sets} observe that we have 
  \[
    \left| \frac{|S|}{N^{d/2}} - 1 \right| + 
    {N^{-\frac{d}{2(1+2d)}}}
    < 
    1
  \]
  for all large \(N\).
  Therefore, taking 
  \[
    \varepsilon =
    \frac{\eta}{200 k}
    \left(  
      \left| \frac{|S|}{N^{d/2}} - 1 \right| + 
      {N^{-\frac{d}{2(1+2d)}}}
    \right)^{3/40k}
  \]
  proves the result.
\end{proof}

\section*{Acknowledgements}

The author is grateful to his thesis supervisors, Jyoti Prakash Saha 
and Saurabh Shrivastava for their support during the preparation of this article.
He acknowledges the fellowship from the Council of Scientific and 
Industrial Research (CSIR), India, file no.~09/1020(17125)/2023-EMR-I,
during the Ph.D. program at IISER Bhopal.

% \bibliography{../mybib}


\end{document}